\documentclass[dvipdfmx]{article} 
\usepackage{amsmath,amssymb,bm,cite,enumerate,url,epstopdf}
\usepackage{graphicx,color,hyperref}
\usepackage[margin=1in]{geometry}
\usepackage[dvipsnames]{xcolor}

\makeatletter
\newenvironment{proof}[1][\proofname]{\par\normalfont
  \topsep6pt plus6pt\trivlist\item[\hskip\labelsep\itshape
  #1\@addpunct{:}]\ignorespaces}{\qed\endtrivlist}
\newcommand{\proofname}{Proof}
\DeclareRobustCommand{\qed}{%
  \ifmmode
  \else\leavevmode\unskip\penalty9999\hbox{}\nobreak\hfill\fi
  \quad\hbox{\qedsymbol}}
\newcommand{\qedsymbol}{\openbox}
\newcommand{\openbox}{\leavevmode\hbox to.77778em{%
    \hfil\vrule\vbox to.675em{%
      \hrule width.6em\vfil\hrule}\vrule\hfil}}
\makeatother

\newcommand{\Prb}{\mathbb{P}}\newcommand{\Exp}{\mathbb{E}}
\newcommand{\N}{\mathbb{N}}\newcommand{\R}{\mathbb{R}}
\newcommand{\B}{\mathcal{B}}\newcommand{\F}{\mathcal{F}}
\newcommand{\G}{\mathcal{G}}
\newcommand{\dd}{\mathrm{d}}
\newcommand{\os}{\mathrm{o}}
\newcommand{\SINR}{\mathsf{SINR}}
\newcommand{\CP}{\mathsf{CP}}
\newcommand{\Tx}{\mathsf{Tx}}
\newcommand{\convolution}{\mathop{\times\rlap{\kern-.747918em$+$\hss}}}
\newcommand{\ind}[1]{\bm{1}_{#1}}
\let\Bar\overline\let\Tilde\widetilde\let\Hat\widehat

\newtheorem{corollary}{Corollary}
\newtheorem{lemma}{Lemma}
\newtheorem{proposition}{Proposition}
\newtheorem{remark}{Remark}
\newtheorem{theorem}{Theorem}

\title{Neveu's exchange formula for analysis of\\
  wireless networks with hotspot clusters}
\author{Naoto~Miyoshi\thanks{The author is with Department of
    Mathematical and Computing Science, Tokyo Institute of Technology,
    Tokyo 152-8552, Japan.   
    E-mail: miyoshi@is.titech.ac.jp}%
\thanks{The author's work was supported by the Japan Society
  for the Promotion of Science (JSPS) Grant-in-Aid for Scientific
  Research (C) 19K11838.}%
}\date{}%

\begin{document}\sloppy\allowbreak\allowdisplaybreaks
\maketitle

\begin{abstract}
Theory of point processes, in particular Palm calculus within the
stationary framework, plays a fundamental role in the analysis of
spatial stochastic models of wireless communication networks.
Neveu's exchange formula, which connects the respective Palm
distributions for two jointly stationary point processes, is known as
one of the most important results in the Palm calculus.
However, its use in the analysis of wireless networks seems to be
limited so far and one reason for this may be that the formula in a
well-known form is based upon the Voronoi tessellation.
In this paper, we present an alternative form of Neveu's exchange
formula, which does not rely on the Voronoi tessellation but includes
the one as a special case.
We then demonstrate that our new form of the exchange formula is
useful for the analysis of wireless networks with hotspot clusters
modeled using cluster point processes.\\
\textbf{Keywords:}
Stationary point processes, Palm calculus, Neveu's exchange formula,
cluster point processes, device-to-device networks, hotspot clusters,
coverage probability, device discovery.
\end{abstract}

\section{Introduction}

Spatial stochastic models have been widely accepted in the literature
as mathematical models for the analysis of wireless communication
networks, where irregular locations of wireless nodes, such as base
stations~(BSs) and user devices, are modeled using spatial point
processes on the Euclidean plane (see, e.g.,
\cite{BaccBlas09a,BaccBlas09b,HaenGant09,Haen13,Mukh14,BlasHaenKeelMukh18}
for monographs and
\cite{AndrGuptDhil16,ElSaSultAlouWin17,HmamBenjSaouYaniDiRe21,LuSaleHaenHossJian21}
for recent survey and tutorial articles).
In such analysis of wireless networks, the theory of point processes,
in particular Palm calculus within the stationary framework, plays a
fundamental role.
Neveu's exchange formula, which connects the respective Palm
distributions for two jointly stationary point processes, is known as
one of the most important results in the Palm calculus.
However, its use in the analysis of wireless networks seems to be
limited so far and one reason for this may be that the formula in a
well-known form is based upon the Voronoi tessellation~(see, e.g,
\cite[Section~6.3]{BaccBlasKarr20}).
In this paper, we present an alternative form of Neveu's exchange
formula, which does not rely on the Voronoi tessellation but includes
the one as a special case, and then demonstrate that it is useful for
the analysis of spatial stochastic models based on cluster point
processes.

A cluster point process represents a state such that there exist a
large number of clusters consisting of multiple points and is used to
model the locations of wireless nodes in an (urban) area with a number
of hotspots.
Indeed, many researchers have adopted the cluster point processes in
their models of various wireless networks such as ad~hoc
networks \cite{GantHaen09}, heterogeneous networks
\cite{SuryMollFett15,ChunHasnGhra15,SahaAfshDhil17,SahaAfshDhil18,AfshDhil18,SahaDhilMiyoAndr19,YangLimZhaoMota21},
device-to-device~(D2D) networks \cite{AfshDhilChon16}, wireless
powered networks \cite{ChenWangZhang17}, unmanned aerial vehicle
assisted networks \cite{TurgGurs18}, and so on.
In this paper, we focus on so-called stationary Poisson-Poisson
cluster processes (PPCPs) (see, e.g., \cite{BlasYoge09,Miyo19}) and
apply the new form of the exchange formula to the analysis of
stochastic models based on them.

We first use the exchange formula for the Palm characterization,
where we derive the intensity measure, the generating functional and
the nearest-neighbor distance distribution for a stationary PPCP under
its Palm distribution.
Although these results are known in the literature (see, e.g,
\cite{Baud81,GantHaen09}), we here give them simple and unified
proofs using the new form of the exchange formula.
We next consider some applications to wireless networks modeled using
stationary PPCPs, where we examine the problems of coverage and device
discovery in a D2D network.
The coverage analysis of a D2D network model based on a cluster point
process was considered in \cite{AfshDhilChon16}, where a device
communicates with another device in the same cluster.
In contrast to this, we assume here that a device receives messages
from the nearest transmitting device, which is possibly in a different
cluster because clusters may overlap in space.
For this model, we derive the coverage probability using the exchange
formula.
On the other hand, in the problem of device discovery, transmitting
devices transmit broadcast messages and a receiving device can detect
the transmitters if it can successfully decode the broadcast messages.
Such a problem was studied in
\cite{HamiChelBussFleu08,Bacc6151335,KwonChoi14} when the devices are
located according to a homogeneous Poisson point process~(PPP) and in
\cite{KwonJuLee20} when the devices are located according to a
Ginibre point process (see, e.g., \cite{MiyoShir14a,MiyoShir16} for
the Ginibre point process and its applications to wireless networks).
We consider the case where the devices are located according to a
stationary PPCPs and derive the expected number of transmitting
devices discovered by a receiving device.
We should note that Neveu's exchange formula is also introduced in a
more general form in \cite{LastThor09,Last10,GentLast11}, so that the
form presented in the paper may be within its scope.
Nevertheless, we see in the rest of the paper that our new form would
be valuable and could spread the application fields of the exchange
formula.

The rest of the paper is organized as follows.
The new form of Neveu's exchange formula is derived in the next
section, where the relations with the existing forms are also
discussed.
In Section~\ref{sec:PPCP}, the exchange formula is applied to the Palm
characterization of a stationary PPCP, where alternative proofs of the
intensity measure, the generating functional and the nearest-neighbor
distance distribution under the Palm distribution are given.
In section~\ref{sec:WNs}, some applications to wireless network models
are examined, where for a D2D network model based on a stationary
PPCP, the coverage probability and the expected number of discovered
devices are derived using the exchange formula. 
The results of numerical experiments are also presented there.
Concluding remarks are provided in Section~\ref{sec:concl}.

\section{Neveu's exchange formula}

In this section, we discuss point processes on the $d$-dimensional
Euclidean space $\R^d$ within the stationary framework (see, e.g.,
\cite[Chapter~6]{BaccBlasKarr20} for details on the stationary
framework).
In what follows, $\B(\R^d)$ denotes the Borel $\sigma$-field on $\R^d$
and $\delta_x$ denotes the Dirac measure with mass at $x\in\R^d$.
Let $(\Omega,\F,\Prb)$ denote a probability space.
On $(\Omega,\F)$, a flow $\{\theta_t\}_{t\in\R^d}$ is defined such
that $\theta_t$:~$\Omega\to\Omega$ is $\F$-measurable and bijective
satisfying $\theta_t\circ\theta_u = \theta_{t+u}$ for $t,u\in\R^d$, 
where $\theta_0$ is the identity for $0=(0,0,\ldots,0)\in\R^d$; so
that $\theta_t^{-1} = \theta_{-t}$ for $t\in\R^d$.
We assume that the probability measure $\Prb$ is invariant to the flow
$\{\theta_t\}_{t\in\R^d}$ (in other words, $\{\theta_t\}_{t\in\R^d}$
preserves $\Prb$) in the sense that $\Prb\circ\theta_t^{-1} = \Prb$
for any $t\in\R^d$, where $\theta_t^{-1}(A) = \{\omega\in\Omega :
\theta_t(\omega)\in A\}$ for $A\in\F$.
A point process $\Phi = \sum_{n=1}^\infty\delta_{X_n}$ on $\R^d$ is
said to be compatible with the flow $\{\theta_t\}_{t\in\R^d}$ if it
holds that $\Phi(B)\circ\theta_t = \Phi(\theta_t(\omega), B) =
\Phi(\omega, B+t) = \Phi(B+t)$ for $\omega\in\Omega$, $B\in\B(\R^d)$
and $t\in\R^d$, where $B+t = \{x+t \in \R^d : x \in B\}$; that is, for
$t\in\R^d$ and $n\in\N=\{1,2,\ldots\}$, there exists an $n'\in\N$ such
that $X_n\circ\theta_t = X_{n'}-t$.
Under the assumption of the $\{\theta_t\}_{t\in\R^d}$-invariance of
$\Prb$, a point process $\Phi$ compatible with
$\{\theta_t\}_{t\in\R^d}$ is stationary in $\Prb$ and furthermore, two
point processes $\Phi$ and $\Psi$, both of which are compatible with
$\{\theta_t\}_{t\in\R^d}$, are jointly stationary in $\Prb$.

Let $\Phi = \sum_{n=1}^\infty\delta_{X_n}$ and $\Psi =
\sum_{m=1}^\infty\delta_{Y_m}$ denote point processes on $\R^d$, which
are both simple, compatible with $\{\theta_t\}_{t\in\R^d}$ and have
positive and finite intensities $\lambda_\Phi$ and $\lambda_\Psi$,
respectively.
Thus, $\Phi$ and $\Psi$ are jointly stationary in probability $\Prb$
and the respective Palm probabilities $\Prb_\Phi^0$ and $\Prb_\Psi^0$
are well-defined.
Note that $\Prb_\Phi^0(\Phi(\{0\})=1) = \Prb_\Psi^0(\Psi(\{0\})=1) =
1$.
In this paper, when we consider the event $\{\Phi(\{0\})=1\}\in\F$, we
assign index~$0$ to the point at the origin; that is, $X_0=0$ on
$\{\Phi(\{0\})=1\}$, and this is also the case for $\Psi$; that is,
$Y_0=0$ on $\{\Psi(\{0\})=1\}$.
To present an alternative form of Neveu's exchange formula, we
introduce a family of shift operators $S_t$, $t\in\R^d$, on the set of
measures $\eta$ on $(\R^d,\B(\R^d))$ by $S_t\eta(B) = \eta(B+t)$ for
$B\in\B(\R^d)$.
For example, operating $S_t$ on the point process $\Psi =
\sum_{m=1}^\infty\delta_{Y_m}$, we have $S_t\Psi =
\sum_{m=1}^\infty\delta_{Y_m-t} = \Psi\circ\theta_t$. 
The shift operators $S_t$, $t\in\R^d$, also work on a function $h$ on
$\R^d$ such as $S_t h(x) = h(x+t)$ for $x\in\R^d$.

\begin{theorem}\label{thm:Neveu}
For the two jointly stationary point processes $\Phi =
\sum_{n=1}^\infty \delta_{X_n}$ and $\Psi = \sum_{m=1}^\infty
\delta_{Y_m}$ described above, we assume that a family of point
processes $\Psi_n = \sum_{k=1}^{\kappa_n}\delta_{Y_{n,k}}$, $n\in\N$,
can be constructed such that
\begin{enumerate}
\item\label{itm:main2} $S_{-X_n}\Psi_n =
  \sum_{k=1}^{\kappa_n}\delta_{X_n+Y_{n,k}}$, $n\in\N$, form a
  partition of $\Psi$; that is, $\Psi = \sum_{n=1}^\infty
  S_{-X_n}\Psi_n$.
\item\label{itm:main1} $\Tilde{\Phi} =
  \sum_{n=1}^\infty\delta_{(X_n,\Psi_n)}$ is a stationary marked point
  process with the set of counting measures on $\R^d$ as its mark
  space.
\end{enumerate}
Then, for any nonnegative random variable $W$ defined on
$(\Omega,\F)$,
\begin{equation}\label{eq:Neveu}
  \lambda_\Psi\,\Exp_\Psi^0[W]
  = \lambda_\Phi\,\Exp_\Phi^0\biggl[
      \int_{\R^d} W\circ\theta_y\,\Psi_0(\dd y)
    \biggr]
  = \lambda_\Phi\,\Exp_\Phi^0\Biggl[
      \sum_{k=1}^{\kappa_0} W\circ\theta_{Y_{0,k}}
    \Biggr],
\end{equation}
where $\Exp_\Phi^0$ and $\Exp_\Psi^0$ denote the expectations with
respect to the Palm probabilities $\Prb_\Phi^0$ and $\Prb_\Psi^0$,
respectively, and $\Psi_0 = \sum_{k=1}^{\kappa_0}\delta_{Y_{0,k}}$
denotes the mark associated with the point $X_0=0$ on
$\{\Phi(\{0\})=1\}$.
\end{theorem}

\begin{proof}
As with the proof of the exchange formula in
\cite[Theorem~6.3.7]{BaccBlasKarr20}, we start our proof with the
mass transport formula (see, e.g.,
\cite[Theorem~6.1.34]{BaccBlasKarr20}); that is, for any measurable
function $\xi$:~$\Omega\times\R^d\to\Bar{\R}_+$,
\begin{equation}\label{eq:MassTransport}
  \lambda_\Phi\,\Exp_\Phi^0\biggl[
    \int_{\R^d} \xi(y)\,\Psi(\dd y)
  \biggr]
  = \lambda_\Psi\,\Exp_\Psi^0\biggl[
      \int_{\R^d} \xi(-x)\circ\theta_x\,\Phi(\dd x)
    \biggr].
\end{equation}
Let $\xi(y) = W\circ\theta_y\,\Psi_0(\{y\})$ on $\{\Phi(\{0\})=1\}$.
Then, the left-hand side of \eqref{eq:MassTransport} becomes
\[
  \lambda_\Phi\,\Exp_\Phi^0\biggl[
    \int_{\R^d} W\circ\theta_y\,\Psi_0(\{y\})\,\Psi(\dd y)
  \biggr]
  = \lambda_\Phi\,\Exp_\Phi^0\biggl[
      \int_{\R^d} W\circ\theta_y\,\Psi_0(\dd y)
  \biggr].
\]
On the other hand, the right-hand side of \eqref{eq:MassTransport} is
reduced to
\[
  \lambda_\Psi\,\Exp_\Psi^0\biggl[
    \int_{\R^d}
      \bigl(W\circ\theta_{-x}\,\Psi_0(\{-x\})\bigr)\circ\theta_x\,
    \Phi(\dd x)
  \biggr]
  = \lambda_\Psi\,\Exp_\Psi^0\biggl[
      W\sum_{n=1}^\infty S_{-X_n}\Psi_n(\{0\})
    \biggr]
  = \lambda_\Psi\,\Exp_\Psi^0[W],
\]
where the first equality follows from $\Psi_0\circ\theta_{X_n} =
\Psi_n$ for $n\in\N$ and the fact that the point of $\Psi$ at the
origin on a sample $\omega\in\{\Psi(\{0\})=1\}$ is shifted to
location $-x$ on the shifted sample $\theta_x(\omega)$ for $x\in\R^d$,
and the last equality follows since there exists exactly one point,
say $X_n$, of $\Phi$ such that its mark
$\Psi_n=\sum_{k=1}^{\kappa_n}\delta_{Y_{n,k}}$ has a point, say
$Y_{n,k}$, satisfying $X_n+Y_{n,k}=0$ on $\{\Psi(\{0\})=1\}$.
The proof is completed.
\end{proof}

\begin{remark}
Let $W\equiv1$ in \eqref{eq:Neveu}.
Then, we have $\Exp_\Phi^0[\kappa_0] = \lambda_\Psi/\lambda_\Phi$ and
therefore, each $\Psi_n$ in Theorem~\ref{thm:Neveu} has finite points.
In the form of the exchange formula in \cite[Theorems~6.3.7 \&
  6.3.19]{BaccBlasKarr20}, the point process $\Psi$ is partitioned by
the Voronoi tessellation for $\Phi$, which corresponds to a special
case of \eqref{eq:Neveu} such that $S_{-X_n}\Psi_n(\cdot) =
\Psi(\cdot\cap V_\Phi(X_n))$ for $n\in\N$, where $V_\Phi(X_n)$ denotes
the Voronoi cell of point $X_n$ of $\Phi$.
The condition in \cite{BaccBlasKarr20} such that there are no points
of $\Psi$ on the boundaries of Voronoi cells~$V_\Phi(X_n)$, $n\in\N$,
is covered by our Condition~\ref{itm:main2} in
Theorem~\ref{thm:Neveu}, where $S_{-X_n}\Psi_n$, $n\in\N$, form a
partition of $\Psi$ and have no common points.
On the other hand, our Theorem~\ref{thm:Neveu} considers only the case
where the point process $\Phi$ is simple unlike
\cite[Theorem~6.3.7]{BaccBlasKarr20}.
However, this would be enough for applications to wireless networks
and it could be extended to the non-simple case.
It should also be noted that, in \cite{LastThor09,Last10,GentLast11},
a more general formula is introduced under the name of Neveu's
exchange formula, from which the mass transport
formula~\eqref{eq:MassTransport} is derived.
In that sense, our form~\eqref{eq:Neveu} may be within its scope.
Nevertheless, we can see in the following sections that
Theorem~\ref{thm:Neveu} is valuable in the sense that it is tractable
and can spread the application fields of the exchange formula.
\end{remark}

\section{Applications to cluster point processes}\label{sec:PPCP}

In this section, we demonstrate that Neveu's exchange
formula~\eqref{eq:Neveu} in Theorem~\ref{thm:Neveu} is useful to
characterize the Palm distribution of stationary cluster point
processes.
A cluster point process is, roughly speaking, constructed by placing
point processes (usually with finite points), called offspring
processes, around respective points of another point process, called a
parent process, and is used to represent a state such that there exist
a large number of clusters consisting of multiple points (see
Figure~\ref{fig:clusterPP}).
In particular, we focus here on a stationary PPCP described next.

\begin{figure}
\begin{center}
\frame{\vbox{\kern1cm\hbox{\kern-1cm\includegraphics[width=9cm, trim= 55 10 60 20,
    clip]{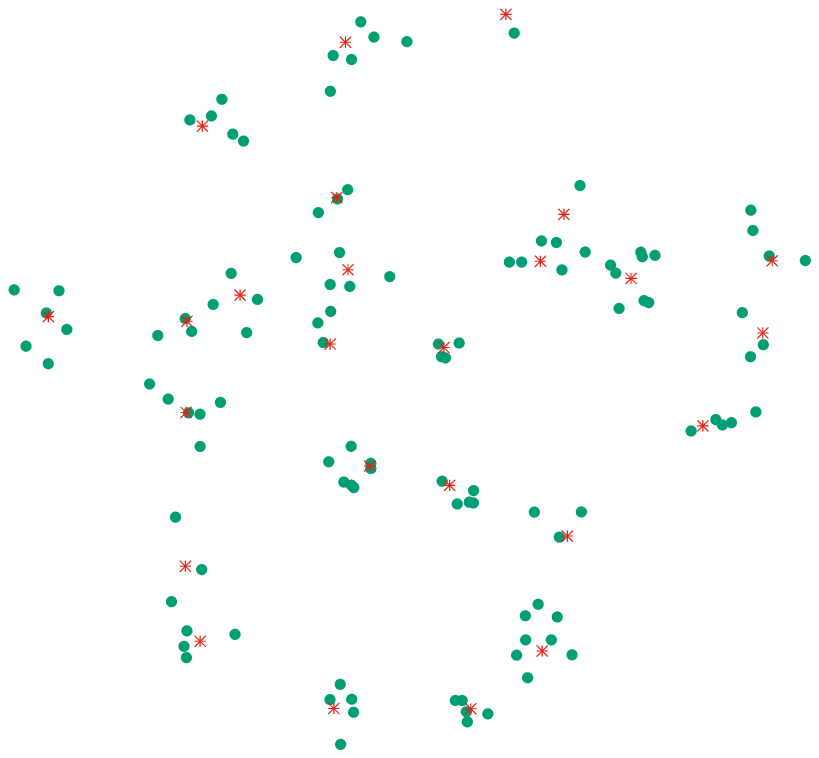}}\kern-1cm}}%
\end{center}
\caption{A sample of a $2$-dimensional cluster point process
  (\textcolor{ForestGreen}{$\bullet$} represents the points of the
  cluster point process and \textcolor{red}{$\convolution$} represents
  the points of the parent process)}\label{fig:clusterPP}
\end{figure}

\subsection{Poisson-Poisson cluster processes}\label{subsec:PPCP}

Let $\Phi = \sum_{n=1}^\infty\delta_{X_n}$ denote a homogeneous PPP on
$\R^d$, which works as the parent process, and let $\Psi_n =
\sum_{k=1}^{\kappa_n}\delta_{Y_{n,k}}$, $n\in\N$, denote a family of
finite (therefore inhomogeneous) and mutually independent PPPs on
$\R^d$, which are also independent of $\Phi$ and work as the offspring
processes.
Then, PPCP $\Psi = \sum_{m=1}^\infty\delta_{Y_m}$ is given as
\begin{equation}\label{eq:PPCP}
  \Psi = \sum_{n=1}^\infty S_{-X_n}\Psi_n
  = \sum_{n=1}^\infty\sum_{k=1}^{\kappa_n}\delta_{X_n+Y_{n,k}}.
\end{equation}
The PPCP $\Psi$ constructed as above is stationary since the parent
process $\Phi$ is stationary and the offspring processes $\Psi_n$,
$n\in\N$, are independent and identically distributed (see, e.g.,
\cite[Example~2.3.18]{BaccBlasKarr20}).
We assume that $\Phi$ has a positive and finite intensity
$\lambda_\Phi$, and $\Psi_n$, $n\in\N$, have an identical intensity
measure $\Lambda_{\os} = \mu\,Q$, where $\mu$ is a positive constant
and $Q$ is a probability distribution on $(\R^d,\B(\R^d))$.
Thus, the number of points in each offspring process follows a Poisson
distribution with mean $\mu$, so that the intensity of $\Psi$ is equal
to $\lambda_\Psi = \lambda_\Phi\mu$, and offspring points are
scattered on $\R^d$ according to $Q$ independently of each other.
We further assume that $Q$ is diffuse; that is, $Q(\{x\})=0$ for any
$x\in\R^d$, to make $\Psi$ simple.
We refer to $S_{-X_n}\Psi_n$ in~\eqref{eq:PPCP} as the cluster
associated with $X_n$ for $n\in\N$.
Two main examples of the PPCPs are the (modified) Thomas point process
and the Mat{\'e}rn cluster process (see, e.g.,
\cite[Example~5.5]{ChiuStoyKendMeck13}).
When $Q$ is an isotropic normal distribution, then the obtained PPCP
is called the Thomas point process.
On the other hand, when $Q$ is the uniform distribution on a fixed
ball centered at the origin, then the result is called the Mat{\'e}rn
cluster process.
Note that the PPCP $\Psi$ and its parent process $\Phi$ satisfy the
conditions of Theorem~\ref{thm:Neveu}.

\subsection{Characterization of Palm distribution}

For a stationary point process $\Psi$, let $\Psi^! := \Psi-\delta_0$
on the event $\{\Psi(\{0\})=1\}$, which is referred to as the reduced
Palm version of $\Psi$.

\begin{lemma}\label{lem:intensity}
For the stationary PPCP $\Psi$ described in Section~\ref{subsec:PPCP},
the intensity measure of the reduced Palm version $\Psi^!$ (with
respect to the Palm distribution) is given by
\begin{equation}\label{eq:intensity}
  \Lambda_\Psi^0(B)
  := \Exp_\Psi^0[\Psi^!(B)]
  = \lambda_\Phi\,\mu\,|B|
    + \mu\int_{\R^d}Q(B-y)\,Q^-(\dd y),
  \quad B\in\B(\R^d),
\end{equation}
where $|\cdot|$ denotes the Lebesgue measure on $(\R^d,\B(\R^d))$ and
$Q^-(B)=Q(-B)$ with $-B = \{-x : x\in B\}$ for $B\in\B(\R^d)$.
\end{lemma}

\begin{proof}
Since the offspring processes $\Psi_n$, $n\in\N$, are PPPs, the PPCP
$\Psi$ is a Cox point process; that is, once the parent
process $\Phi = \sum_{n=1}^\infty\delta_{X_n}$ is given, $\Psi$ is
conditionally an inhomogeneous PPP with a conditional intensity
measure $\mu\sum_{n=1}^\infty S_{-X_n}Q$ (see, e.g.,
\cite[Example~2.3.13]{BaccBlasKarr20}).
Since the reduced Palm version of a PPP is identical in distribution
to its original version (not conditioned on $\{\Psi(\{0\})=1\}$) by
Slivnyak's theorem (see, e.g.,
\cite[Proposition~13.1.VII]{DaleVere08} or
\cite[Theorem~3.2.4]{BaccBlasKarr20}), we have
\[
  \Exp_\Psi^0[\Psi^!(B) \mid \Phi]
  = \Exp[\Psi(B)\mid \Phi]
  = \mu\sum_{n=1}^\infty Q(B-X_n),
  \quad B\in\B(\R^d).
\]
Taking the expectation with respect to $\Prb_\Psi^0$ and then applying
Theorem~\ref{thm:Neveu}, we obtain
\begin{align*}
  \Exp_\Psi^0[\Psi^!(B)]
  &= \mu\,\Exp_\Psi^0\Biggl[
       \sum_{n=1}^\infty Q(B-X_n)
     \Biggr]
   = \Exp_\Phi^0\Biggl[
       \int_{\R^d}
         \Biggl(\sum_{n=1}^\infty Q(B-X_n)\Biggr)\circ\theta_y\,
       \Psi_0(\dd y)
     \Biggr]
  \nonumber\\
  &= \Exp_\Phi^0\Biggl[
       \int_{\R^d}
         \sum_{n=0}^\infty Q(B-X_n+y)\,
       \Psi_0(\dd y)
     \Biggr] 
  \nonumber\\
  &= \mu\int_{\R^d}
       \Biggl(
         Q(B+y)
         + \Exp_\Phi^0\Biggl[
             \sum_{n=1}^\infty Q(B-X_n+y)
           \Biggr]
       \Biggr)\,
     Q(\dd y),
\end{align*}
where $\lambda_\Psi = \lambda_\Phi\mu$ is used in the second equality,
the third equality follows because, for any $n\in\N$ and $y\in\R^d$,
there exists an $n'\in\N\cup\{0\}$ such that $X_n\circ\theta_y =
X_{n'}-y$ on $\{\Phi(\{0\})=1\}$, and in the last equality, we apply
Campbell's formula (see, e.g., \cite[Proposition~2.7]{LastPenr17} or
\cite[Theorem~1.2.5]{BaccBlasKarr20}) for $\Psi_0$.
For the expectation in the last expression above, Slivnyak's theorem,
Campbell's formula for $\Phi$ and then Fubini's theorem yield
\begin{align*}
  \Exp_\Phi^0\Biggl[
    \sum_{n=1}^\infty Q(B-X_n+y)
  \Biggr]  
  &= \lambda_\Phi\int_{\R^d}Q(B-x)\,\dd x
   = \lambda_\Phi\int_{\R^d}\int_{B-x}Q(\dd z)\,\dd x
  \\
  &= \lambda_\Phi\int_{\R^d}\int_{B-z}\dd x\,Q(\dd z)
   = \lambda_\Phi\,|B|,
\end{align*}
which completes the proof.
\end{proof}

\begin{remark}\label{rem1}
The second term on the right-hand side of \eqref{eq:intensity} is of
course equal to $\mu\int Q(B+y)\,Q(\dd y)$.
We adopt the form in Lemma~\ref{lem:intensity} due to its
interpretability.
Since $Q$ is the distribution for the position of an offspring point
viewed from its parent, $Q^-$ represents the distribution for the
location of the parent of the offspring point at the origin on the
event $\{\Psi(\{0\})=1\}$.
On the other hand, $\mu\,Q(B-y)$ gives the expected number of
offspring points falling in $B\in\B(\R^d)$ among a cluster whose
parent is shifted to $y\in\R^d$.
In other words, the second term on the right-hand side of
\eqref{eq:intensity} represents the expected number of offspring
points falling in $B$ among the cluster which is given to have one
point at the origin.
Since the first term on the right-hand side of \eqref{eq:intensity} is
equal to $\Lambda_\Psi(B) = \Exp[\Psi(B)]$,
Lemma~\ref{lem:intensity} states that the intensity measure for the
reduced Palm version of a stationary PPCP is given as the sum of the
intensity measure of the stationary version and that of a cluster
which has one point at the origin.
Lemma~\ref{lem:intensity} is also a slight generalization of the
result in \cite[Section~2.2]{TanaOgatStoy08}.
\end{remark}

The observation in Remark~\ref{rem1} is further enhanced by the
following proposition.

\begin{proposition}\label{prp:PGFL}
For the stationary PPCP $\Psi = \sum_{m=1}^\infty\delta_{Y_m}$
described in Section~\ref{subsec:PPCP}, the generating functional of
the reduced Palm version $\Psi^!$ (with respect to the Palm
distribution) is given by
\begin{equation}\label{eq:PalmPGFL}
  \G_\Psi^0(h)
  := \Exp_\Psi^0\Biggl[\prod_{m=1}^\infty h(Y_m)\Biggr]
  = \G_\Psi(h)\int_{\R^d}\Tilde{h}(z)\,Q^-(\dd z),
\end{equation}
for any measurable function $h$:~$\R^d\to[0,1]$, where $\G_\Psi$ is
the generating functional of the stationary version of $\Psi$ given as
\begin{equation}\label{eq:stationaryPGFL}
  \G_\Psi(h):=
  \Exp\Biggl[\sum_{m=1}^\infty h(Y_m)\Biggr]
  = \G_\Phi(\Tilde{h})
  = \exp\biggl(
      - \lambda_\Phi\int_{\R^d}
          [1-\Tilde{h}(x)]\,
        \dd x
    \biggr),
\end{equation}
and $\Tilde{h}(x)$ denotes the generating
functional of an offspring process $\Psi_1$ whose parent is shifted to
$x\in\R^d$;
\begin{equation}\label{eq:osPGFL}
  \Tilde{h}(x)
  = \G_{\Psi_1}(S_xh)
  = \exp\biggl(
      - \mu\int_{\R^d}
          \bigl[1 - h(x+y)\bigr]\,Q(\dd y)
    \biggr).
\end{equation}
\end{proposition}

Note that in Proposition~\ref{prp:PGFL} above, $\G_\Phi$ is the
generating functional of the parent process~$\Phi$.
The relation $\G_\Psi(h)=\G_\Phi(\Tilde{h})$ with $\Tilde{h}(x) =
\G_{\Psi_1}(S_x h)$ in \eqref{eq:stationaryPGFL} and \eqref{eq:osPGFL}
is known to hold for more general cluster point processes (see, e.g.,
\cite[Example~6.3(a)]{DaleVere03} or \cite[Proposition~2.3.12 \&
  Lemma~2.3.20]{BaccBlasKarr20}), whereas the last equalities in
\eqref{eq:stationaryPGFL} and \eqref{eq:osPGFL} follow because $\Phi$
and $\Psi_1$ are PPPs, respectively (see, e.g., \cite[Exercise
  3.6]{LastPenr17} or \cite[Corollary   2.1.5]{BaccBlasKarr20}).
The relation~\eqref{eq:PalmPGFL} is derived in
\cite[Lemma~1]{GantHaen09}, to which we give another proof using the
exchange formula in Theorem~\ref{thm:Neveu}.

\begin{proof}
As stated in the proof of Lemma~\ref{lem:intensity}, once the parent
process~$\Phi = \sum_{n=1}^\infty\delta_{X_n}$ is given, the
PPCP~$\Psi$ is conditionally an inhomogeneous PPP with the conditional
intensity measure $\mu\sum_{n=1}^\infty S_{-X_n}Q$.
Since the reduced Palm version of a PPP is identical in distribution
to its original (not conditioned) version, we have
\begin{align*}
  \Exp_\Psi^0\Biggl[
    \prod_{m=1}^\infty h(Y_m)\Biggm| \Phi
  \Biggr]
  &= \Exp\Biggl[
       \prod_{m=1}^\infty h(Y_m)\biggm| \Phi
       \Biggr]
  \\
  &= \exp\biggl(
       - \mu\sum_{n=1}^\infty
           \int_{\R^d}(1-h(y))\,Q(\dd y -X_n)
     \biggr)
   = \prod_{n=1}^\infty
       \Tilde{h}(X_n),
\end{align*}
where the generating functional of a PPP is applied in the second
equality.
Taking the expectation with respect to $\Prb_\Psi^0$ and then applying
Theorem~\ref{thm:Neveu}, we obtain
\begin{align*}
  \G_\Psi^0(h)
  &= \Exp_\Psi^0\Biggl[
       \prod_{n=1}^\infty \Tilde{h}(X_n)
     \Biggr]
   = \frac{1}{\mu}\,
     \Exp_\Phi^0\Biggl[
       \int_{\R^d}\prod_{n=0}^\infty
         \Tilde{h}(X_n-z)\,\Psi_0(\dd z)
     \Biggr]
  \\
  &= \int_{\R^d}
       \Tilde{h}(-z)\,
       \Exp_\Phi^0\Biggl[\prod_{n=1}^\infty \Tilde{h}(X_n-z)\Biggr]\,
     Q(\dd z),
\end{align*}
where Campbell's formula for $\Psi_0$ is applied in the last equality.
By Slivnyak's theorem and the stationarity for $\Phi$, we have
$\Exp_\Phi^0\bigl[\prod_{n=1}^\infty \Tilde{h}(X_n-z)\bigr] =
\Exp\bigl[\prod_{n=1}^\infty \Tilde{h}(X_n)\bigr] =
\G_\Phi(\Tilde{h})$, which completes the proof.
\end{proof}

\begin{remark}
The right-hand side of \eqref{eq:PalmPGFL} is given as the generating
functional $\G_\Psi(h)$ of the stationary version of $\Psi$ multiplied
by the integral term $\int\Tilde{h}(z)\,Q^-(\dd z)$.
Since $\Tilde{h}(z)$ represents the generating functional of an
offspring process whose parent is shifted to $z\in\R^d$ and $Q^-$ is
the distribution of the location of the parent point of the offspring
at the origin on the event $\{\Psi(\{0\})=1\}$, this integral term
represents the generating functional of the cluster which is given to
have a point at the origin.
In other words, Proposition~\ref{prp:PGFL} implies that, for a
stationary PPCP, its Palm version is obtained by the superposition of
the original stationary version and an additional independent
offspring process whose parent is placed such that it has an offspring
point at the origin.
This observation is already found in, e.g., \cite{SahaDhilMiyoAndr19}
and is also interpreted such that a point~$z\in\R^d$ is first sampled
from the distribution~$Q^-$ and the Palm version of $\Phi$ at $z$ is
then obtained as $\Phi+\delta_z$ by Slivnyak's theorem, which works as
a parent process of the Palm version of $\Psi$.
Proposition~\ref{prp:PGFL} indeed supports this interpretation.
\end{remark}  

\subsection{Nearest-neighbor distance distributions}

For a stationary point process~$\Psi$ on $\R^d$, let $\Psi^! =
\sum_{m=1}^\infty \delta_{Y_m}$ be its reduced Palm version on
$\{\Psi(\{0\})=1\}$ and let $Y_*$ denote the nearest point of $\Psi^!$
from the origin.
Then, the nearest-neighbor distance distribution for $\Psi$ is defined
as the probability distribution for $\|Y_*\|$ with respect to
$\Prb_\Psi^0$, where $\|\cdot\|$ denotes the Euclidean distance.
We show below that the nearest-neighbor distance distribution for a
stationary PPCP is obtained in a similar way to
Proposition~\ref{prp:PGFL}.

\begin{proposition}\label{prp:NND}
For the stationary PPCP~$\Psi$ described in Section~\ref{subsec:PPCP},
the complementary nearest-neighbor distance distribution is given by
\begin{equation}\label{eq:NND}
  \Prb_\Psi^0(\|Y_*\| > r)
   = \G_\Phi(h^*_r)
     \int h^*_r(t)\,Q^-(\dd t),\quad r\ge0,
\end{equation}
where $h^*_r(x) = e^{-\mu Q(b_0(r)-x)}$ and $b_0(r)$ denotes a
$d$-dimensional ball centered at the origin with radius~$r$.
\end{proposition}

\begin{proof}
As with the proof of Proposition~\ref{prp:PGFL}, we consider the
conditional probability given the parent process~$\Phi =
\sum_{n=1}^\infty \delta_{X_n}$ and obtain
\begin{align}\label{eq:NND-prf}
  \Prb_\Psi^0(\|Y_*\| > r \mid \Phi)
  &= \Prb_\Psi^0\bigl(
       \Psi^!(b_0(r)) = 0 \mid \Phi
     \bigr)
   = \Prb\bigl(
       \Psi(b_0(r)) = 0 \mid \Phi
     \bigr)
  \nonumber\\       
  &= \prod_{n=1}^\infty e^{- \mu Q(b_0(r)-X_n)}
   = \prod_{n=1}^\infty h^*_r(X_n),
\end{align}
where the second equality follows from Slivnyak's theorem and the
third does because $\Psi^!$ is conditionally an inhomogeneous PPP with
the intensity measure $\mu\sum_{n=1}^\infty S_{-X_n}Q$ when $\Phi$ is
given.
The rest of the proof is similar to that of
Proposition~\ref{prp:PGFL}.
\end{proof}

\begin{remark}
In~\eqref{eq:NND}, the term $\G_\Phi(h^*_r)$ is the complementary
contact distance distribution and that is obtained by taking the
expectation of \eqref{eq:NND-prf} with respect to $\Prb$, instead of
$\Prb_\Psi^0$ (see, e.g., \cite{Miyo19}).
The result of Proposition~\ref{prp:NND} is consistent with the
existing ones in, e.g.,
\cite{Baud81,AfshSahaDhil17b,AfshSahaDhil17,PandDhilGupt20} and gives a
unified approach to derive the nearest-neighbor distance distributions
for stationary PPCPs.
\end{remark}

\section{Applications to wireless networks with hotspot
  clusters}\label{sec:WNs}

In this section, we apply Theorem~\ref{thm:Neveu} to the analysis of a
D2D network with hotspot clusters modeled using a stationary PPCP.
We here suppose $d=2$, but unless otherwise specified, the discussion
holds for $d\ge2$ theoretically.

\subsection{Model of a Device-to-Device Network}\label{subsec:D2D}

Wireless devices are distributed on $\R^d$ according to a stationary
point process $\Psi = \sum_{m=1}^\infty\delta_{Y_m}$.
At each time slot, each device is in transmission mode with
probability~$p\in(0,1)$ or in receiving mode with probability $1-p$
independently of the others (half duplex with random access).
Devices in the transmission mode transmit signals but can not receive
ones, whereas the devices in the receiving mode can receive signals
but can not transmit ones.
We assume that all transmitting devices transmit signals with
identical transmission power (normalized to one) and share a common
frequency spectrum.
The path-loss function representing attenuation of signals with
distance is given by $\ell$ satisfying $\ell(r)\ge0$, $r>0$, and
$\int_\epsilon^\infty\ell(r)\,r^{d-1}\,\dd r<\infty$ for $\epsilon>0$.
We further assume that all wireless links receive Rayleigh fading
effects while we ignore shadowing effects.
We focus on the device at the origin, referred to as the typical
device, under the condition of $\{\Psi(\{0\})=1\}$ and examine whether
the typical device can decode messages from other transmitting
devices.
Let $\Psi_{\Tx} = \sum_{m=1}^\infty\delta_{Y'_m}$ denote the
sub-process of $\Psi$ representing the locations of devices in the
transmission mode and for each $m\in\N$, let $H_m$ denote a random
variable representing the fading effect on signals transmitted from
the device at $Y'_m$, where $H_m$, $m\in\N$, are mutually independent,
independent of $\Psi_{\Tx}$ and exponentially distributed with unit
mean due to the Rayleigh fading.
With this setup, the received signal power by the typical device
amounts to $H_m\,\ell(Y'_m)$ when it receives signals from the device
at $Y'_m$.
Hence, if the typical device is in the receiving mode and communicates
with the transmitting device at $Y'_m$, the
\textit{signal-to-interference-plus-noise ratio} (SINR) is given as
\begin{equation}\label{eq:SINR}
  \SINR_m
  = \frac{H_m\,\ell(\|Y'_m\|)}
         {\sum_{\substack{j=1\\j\ne m}}^\infty
            H_j\,\ell(\|Y'_j\|) + N},
\end{equation}
where $N$ denotes a constant representing noise at the origin.
We suppose that the typical device can successfully decode a message
from the device at $Y'_m$ if the typical device is in the receiving
mode and $\SINR_m$ in~\eqref{eq:SINR} exceeds a predefined
threshold~$\theta>0$.

\subsection{Coverage analysis}\label{subsec:coverage}

We here suppose that a device in the receiving mode communicates
with the nearest device in transmission mode.
The probability that the typical device can successfully decode a
message from its partner is called the coverage probability and is
given by
\begin{equation}\label{eq:coverage-def}
  \CP(\theta)
  = (1-p)\sum_{m=1}^\infty
      \Prb_\Psi^0\bigl(
        \SINR_m > \theta, \|Y'_m\|\le\|Y'_j\|, j\in\N
      \bigr),
\end{equation}
where $1-p$ on the right-hand side indicates that the typical device
must be in the receiving mode and the sum over $m\in\N$ represents the
probability that the SINR from the nearest transmitting device exceeds
the threshold~$\theta$.
We now suppose that the point process~$\Psi$ representing the
locations of devices is given as a stationary PPCP studied in
Section~\ref{sec:PPCP}.
Then, $\Psi_{\Tx} = \sum_{m=1}^\infty\delta_{Y'_m}$ representing the
locations of devices in the transmission mode is also a stationary
PPCP, where the parent process remains the same as the homogeneous
PPP~$\Phi$ with intensity~$\lambda_\Phi$, whereas the offspring
processes~$\Psi'_n = \sum_{k=1}^{\kappa'_n}\delta_{Y'_{n,k}}$,
$n\in\N$, are finite PPPs with the intensity measure~$p\mu Q$.

\begin{theorem}\label{thm:Coverage}
For the model of a D2D network described in Section~\ref{subsec:D2D}
with the devices deployed according to a stationary PPCP in
Section~\ref{subsec:PPCP}, the coverage probability is given by
\begin{equation}\label{eq:coverage}  
  \CP(\theta)
  = (1-p)p\mu
    \int_{\R^d}
      \bigl(I_{1,\theta}(t) + I_{2,\theta}(t)\bigr)\,
    Q^-(\dd t),
\end{equation}
where $Q^-$ is given in Lemma~\ref{lem:intensity} and
\begin{align}\label{eq:C(theta)}
  I_{1,\theta}(t)
  &= \int_{\R^d}
       e^{-\theta N/\ell(\|y\|)}\,
       C_\theta(y, t)
       E_\theta(y)\,
     Q(\dd y - t),
  \nonumber\\
  I_{2,\theta}(t)
  &= \lambda_\Phi
     \int_{\R^d}\!\int_{\R^d}
       e^{-\theta N/\ell(\|y\|)}\,
       C_\theta(y, t)\,
       C_\theta(y, x)\,E_\theta(y)\,
     Q(\dd y - x)\,\dd x,
  \nonumber\\
  E_\theta(y)
  &= \exp\biggl(
       - \lambda_\Phi\int_{\R^d}
           \bigl[1-C_\theta(y,w)\bigr]\,
         \dd w
     \biggr),
  \nonumber\\
  C_\theta(y,x)
  &= \exp\biggl(
       - p\mu\,\biggl[
           1 -
               \int_{\|z\|>\|y\|}
                 \Bigl(
                   1 + \theta\,\frac{\ell(\|z\|)}{\ell(\|y\|)}
                 \Bigr)^{-1}\,
               Q(\dd z - x)
         \biggr]
     \biggr).
\end{align}
\end{theorem}

Before proceeding on the proof of Theorem~\ref{thm:Coverage}, we
give an intuitive interpretation to the result of it.
First, as stated in the preceding section, $Q^-$ denotes the
distribution for the location of the parent point of the typical
device at the origin.
Thus, $p\mu I_{1,\theta}(t)$ and $p\mu I_{2,\theta}(t)$
in~\eqref{eq:coverage} represent the cases where the typical device,
whose parent is located at $t\in\R^d$, communicates with the
transmitting device in the same cluster and in a different cluster,
respectively; that is, the location of the communication partner is
sampled from a finite PPP with the intensity measure $p\mu Q(\dd y -
t)$ in $I_{1,\theta}(t)$ and is from one with $p\mu Q(\dd y -x)$ in
$I_{2,\theta}(t)$, where $x$ is also sampled from a homogeneous PPP
with intensity $\lambda_\Phi$.
Moreover, $E_\theta(y)$ represents the effect from other clusters
which are neither the one having the typical device nor the one having
its communication partner at $y$.
Finally, $C_\theta(y,x)$ represents the effect of the cluster with
the parent point at $x\in\R^d$ when the typical device communicates
with the transmitting device at $y$.

\begin{proof}
Similar to the proof of Proposition~\ref{prp:PGFL}, once the parent
process~$\Phi = \sum_{n=1}^\infty\delta_{X_n}$ is given, the point
process~$\Psi_{\Tx}$ representing the locations of devices in the
transmission mode is conditionally an inhomogeneous PPP with the
conditional intensity measure $p\mu\sum_{n=1}^\infty S_{-X_n}Q$.
Thus, we can use the corresponding approach to that obtaining the
coverage probability for a cellular network with BSs deployed
according to a PPP (see, e.g,, \cite{AndrBaccGant11} or
\cite[Section~5.2]{BlasHaenKeelMukh18}).
Since $H_m$, $m\in\N$, are mutually independent, exponentially
distributed, and also independent of $\Phi$, we have from
\eqref{eq:coverage-def},
\begin{align*}
  &\Prb_\Psi^0\bigl(
     \SINR_m > \theta, \|Y'_m\|\le\|Y'_j\|, j\in\N
   \bigm| \Phi\bigr)
  \\ 
  &= \Prb_\Psi^0\Biggl(
       H_m > \frac{\theta}{\ell(\|Y'_m\|)}\,
             \Biggl(
               \sum_{\substack{j=1\\j\ne m}}^\infty
                 H_j\,\ell(\|Y'_j\|)
               + N
             \Biggr),\:
       \|Y'_m\|\le\|Y'_j\|, j\in\N
     \Biggm| \Phi \Biggr)
  \\
  &= \Exp_\Psi^0\Biggl[
       e^{-\theta N/\ell(\|Y'_m\|)}
       \prod_{\substack{j=1\\j\ne m}}^\infty
         \biggl(
           1 + \theta\,\frac{\ell(\|Y'_j\|)}{\ell(\|Y'_m\|)}\,
         \biggr)^{-1}\,
         \ind{\{\|Y'_j\|>\|Y'_m\|\}}
     \Biggm|\Phi\Biggr],
\end{align*}
where $\ind{A}$ denotes the indicator function for set~$A$ and we use
$\Prb(H_m>a)=e^{-a}$ for $a\ge0$ and $\Exp[e^{-s H_j}] = (1+s)^{-1}$
in the last equality.
Summing the above expression over $m\in\N$, we have from Slivnyak's
Theorem for $\Psi$ conditioned on $\Phi$ and the refined
Campbell formula~(see, e.g., \cite[Theorem 13.2.III]{DaleVere08},
\cite[Theorem~9.1]{LastPenr17} or
\cite[Theorem~3.1.9]{BaccBlasKarr20}),
\begin{align}\label{eq:cover-prf1}
  &\sum_{m=1}^\infty
     \Prb_\Psi^0\bigl(
       \SINR_m > \theta, \|Y'_m\|\le\|Y'_j\|, j\in\N
     \bigm| \Phi\bigr)
  \nonumber\\
  &= p\mu\sum_{n=1}^\infty
       \int_{\R^d}
         e^{- \theta N/\ell(\|y\|)}\,
         \Exp\Biggl[  
           \prod_{j=1}^\infty
             \biggl(
               1 + \theta\,\frac{\ell(\|Y'_j\|)}{\ell(\|y\|)}\,
             \biggr)^{-1}\,
             \ind{\{\|Y'_j\|>\|y\|\}}
         \Biggm|\Phi\Biggr]\,
       Q(\dd y - X_n).
\end{align}
Furthermore, the generating functional of a PPP applying to the above
expectation yields
\begin{align*}
  &\Exp\Biggl[  
     \prod_{j=1}^\infty
       \biggl(
         1 + \theta\,\frac{\ell(\|Y'_j\|)}{\ell(\|y\|)}\,
       \biggr)^{-1}\,
       \ind{\{\|Y'_j\|>\|y\|\}}
   \Biggm|\Phi\Biggr]
  \\
  &= \exp\biggl(
       - p\mu\sum_{i=1}^\infty
           \int_{\R^d}
             \biggl[
               1 - \biggl(
                     1 + \theta\,\frac{\ell(\|z\|)}{\ell(\|y\|)}\,
                   \biggr)^{-1}\,
                   \ind{\{\|z\|>\|y\|\}}
             \biggr]\,
           Q(\dd z - X_i)
     \biggr)
  \\           
  &= \prod_{i=1}^\infty\exp\biggl(
       - p\mu\,
         \biggl[
           1 - \int_{\|z\|>\|y\|}
                 \biggl(
                   1 + \theta\,\frac{\ell(\|z\|)}{\ell(\|y\|)}\,
                 \biggr)^{-1}\,
               Q(\dd z - X_i)
         \biggr]
       \biggr)
   = \prod_{i=1}^\infty C_\theta(y,X_i).         
\end{align*}
Plugging this into \eqref{eq:cover-prf1}, taking the expectation with
respect to $\Prb_\Psi^0$ and then applying Neveu's exchange formula in
Theorem~\ref{thm:Neveu}, we have
\begin{align}\label{eq:cover-prf2}
  \CP(\theta)
  &= (1-p)p\mu\,
     \Exp_\Psi^0\Biggl[
       \sum_{n=1}^\infty\int_{\R^d}
         e^{-\theta N/\ell(\|y\|)}
         \prod_{i=1}^\infty C_\theta(y,X_i)\,
       Q(\dd y-X_n)
     \Biggr]
  \nonumber\\
  &= (1-p)p\,
     \Exp_\Phi^0\Biggl[
       \int_{\R^d}
         \sum_{n=0}^\infty\int_{\R^d}
           e^{-\theta N/\ell(\|y\|)}
           \prod_{i=0}^\infty C_\theta(y,X_i-t)\,
         Q(\dd y-X_n+t)\,
       \Psi_0(\dd t)    
     \Biggr]
  \nonumber\\
  &= (1-p)p\mu\,
     \int_{\R^d}
       \Exp_\Phi^0\Biggl[
         \sum_{n=0}^\infty\int_{\R^d}
           e^{-\theta N/\ell(\|y\|)}
           \prod_{i=0}^\infty C_\theta(y,X_i-t)\,
         Q(\dd y-X_n+t)\,
       \Biggr]\,
     Q(\dd t),
\end{align}
where we note the existence of $X_0=0$ on $\{\Phi(\{0\})=1\}$ in the
second equality and apply Campbell's formula in the third equality.
Noting that $X_0=0$ on $\{\Phi(\{0\})=1\}$, we separate the
expectation in \eqref{eq:cover-prf2} into
\begin{align}\label{eq:cover-prf3}
  &\Exp_\Phi^0\Biggl[
     \sum_{n=0}^\infty\int_{\R^d}
       e^{-\theta N/\ell(\|y\|)}
       \prod_{i=0}^\infty C_\theta(y,X_i-t)\,
     Q(\dd y-X_n+t)\,
   \Biggr]
  \nonumber\\  
  &= \int_{\R^d}
       e^{-\theta N/\ell(\|y\|)}\,
       C_\theta(y,-t)\,
       \Exp_\Phi^0\Biggl[
         \prod_{i=1}^\infty C_\theta(y,X_i-t)
       \Biggr]\,
     Q(\dd y + t)
  \nonumber\\
  &\quad\mbox{}
   + \Exp_\Phi^0\Biggl[
       \sum_{n=1}^\infty\int_{\R^d}
         e^{-\theta N/\ell(\|y\|)}\,
         C_\theta(y,-t)\,C_\theta(y,X_n-t)\,
         \prod_{\substack{i=1\\i\ne n}}^\infty C_\theta(y,X_i-t)\,
       Q(\dd y-X_n+t)\,
      \Biggr],
\end{align}
and consider the two terms on the right-hand side
of~\eqref{eq:cover-prf3} one by one.
For the first term, the generating functional of a PPP yields
\begin{align}\label{eq:cover-prf4}
  (\text{1st term of~\eqref{eq:cover-prf3}})
  &= \int_{\R^d}
       e^{-\theta N/\ell(\|y\|)}\,
       C_\theta(y,-t)\,
       \exp\biggl(
         - \lambda_\Phi\int_{\R^d}
             [1- C_\theta(y,w)]\,
           \dd w
       \biggr)\,
     Q(\dd y + t)
  \nonumber\\
  &= I_{1,\theta}(-t).
\end{align}
On the other hand, applying Campbell's formula and the generating
functional for $\Phi$ to the second term on the right-hand side of
\eqref{eq:cover-prf3}, we have
\begin{align}\label{eq:cover-prf5}
  &(\text{2nd term of~\eqref{eq:cover-prf3}})
  \nonumber\\
  &= \lambda_\Phi\int_{\R^d}
       \int_{\R^d}
         e^{-\theta N/\ell(\|y\|)}\,
         C_\theta(y,-t)\,C_\theta(y,x)\,
         \Exp_\Phi^0\Biggl[
           \prod_{i=1}^\infty C_\theta(y,X_i-t)
         \Biggr]
       Q(\dd y-x)\,
     \dd x
  \nonumber\\
  &= \lambda_\Phi\int_{\R^d}
       \int_{\R^d}
         e^{-\theta N/\ell(\|y\|)}\,
         C_\theta(y,-t)\,C_\theta(y,x)\,
         \exp\biggl(
           - \lambda_\Phi\int_{\R^d}
               [1 - C_\theta(y,w)]\,
             \dd w
         \biggr)\,
       Q(\dd y-x)\,
     \dd x
  \nonumber\\
  &= I_{2,\theta}(-t).
\end{align}
Finally, plugging \eqref{eq:cover-prf4} and \eqref{eq:cover-prf5} into
\eqref{eq:cover-prf3}, and then into \eqref{eq:cover-prf2}, we have
\eqref{eq:coverage} and the proof is completed.
\end{proof}

When $d=2$ and the distribution~$Q$ for the locations of offspring
points depends only on the distance; that is, $Q(\dd y) =
f_{\os}(\|y\|)\,\dd y$ for $y\in\R^2$, we obtain a numerically
computable form of the coverage probability.

\begin{corollary}\label{cor:coverage}
When $d=2$ and $Q(\dd y) = f_{\os}(\|y\|)\,\dd y$, $y\in\R^2$, the
coverage probability in Theorem~\ref{thm:Coverage} is reduced to
\begin{equation}\label{eq:coverage-2d}
  \CP(\theta)
  = 2\pi(1-p)p\mu
    \int_0^\infty\!\!
      \int_0^\infty
        e^{-\theta N/\ell(s)}\,\Hat{E}_\theta(s)\,
        \Hat{C}_\theta(s,u)\,\Hat{I}_\theta(s,u)\,
      \dd s\,
      f_{\os}(u)\,u\,
    \dd u,
\end{equation}
where
\begin{align}
  \Hat{I}_\theta(s,u)
  &= g(s\mid u)
     + 2\pi\lambda_\Phi
       \int_0^\infty
         \Hat{C}_\theta(s,r)\,g(s\mid r)\,r\,
       \dd r,
  \nonumber\\
  \Hat{E}_\theta(s)
  &= \exp\biggl(
       - 2\pi\lambda_\Phi
         \int_0^\infty
           \bigl[1-\Hat{C}_\theta(s,v)\bigr]\,v\,
         \dd v
     \biggr),
  \nonumber\\  
  \Hat{C}_\theta(s,r)
  &= \exp\biggl(
       - p\mu\biggl[
           1 - \int_s^\infty
                 \Bigl(1+\theta\,\frac{\ell(q)}{\ell(s)}\Bigr)^{-1}
                 g(q\mid r)\,
               \dd q 
         \biggr]
     \biggr),
  \label{eq:C(theta)-2d}\\
  g(s\mid r)
  &= 2s \int_0^\pi
          f_{\os}(\sqrt{s^2 + r^2 -2 s r\cos\varphi})\,
        \dd\varphi.
  \nonumber
\end{align}
\end{corollary}

\begin{proof}
Since the distribution $Q$ depends only on the distance, it holds that
$Q^-(\dd t) = Q(\dd t) = f_{\os}(\|t\|)\,\dd t$, $t\in\R^2$, and
\eqref{eq:coverage} is reduced to
\begin{align}
  \CP(\theta)
  &= (1-p)p\mu
     \int_{\R^2}
       \bigl(I_{1,\theta}(t) + I_{2,\theta}(t)\bigr)\,
       f_{\os}(\|t\|)\,
     \dd t
  \nonumber\\
  &= 2\pi(1-p)p\mu
     \int_0^\infty
       \bigl(\Hat{I}_{1,\theta}(u) + \Hat{I}_{2,\theta}(u)\bigr)\,
       f_{\os}(u)\,u\,
     \dd u,
  \label{eq:coverage-cor-prf1}
\end{align}
where the polar coordinate conversion is applied in the second
equality and
\begin{align*}
  \Hat{I}_{1,\theta}(u)
  &= \int_0^\infty
       e^{-\theta N/\ell(s)}\,\Hat{C}_\theta(s,u)\,\Hat{E}_\theta(s)\,
       g(s\mid u)\,
     \dd s,
  \\ 
  \Hat{I}_{2,\theta}(u)
  &= 2\pi\lambda_\Phi
     \int_0^\infty\!\!\int_0^\infty
       e^{-\theta N/\ell(s)}\,\Hat{C}_\theta(s,u)\,
       \Hat{C}_\theta(s,r)\,\Hat{E}_\theta(s)\,
       g(s\mid r)\,
     \dd s\,r\,\dd r.
\end{align*}
Therefore, we have
\[
  \Hat{I}_{1,\theta}(u) + \Hat{I}_{2,\theta}(u)
  = \int_0^\infty
      e^{-\theta N/\ell(s)}\,\Hat{C}_\theta(s,u)\,\Hat{E}_\theta(s)\,
      \Hat{I}_\theta(s,u)\,
    \dd s.
\]
Plugging this into \eqref{eq:coverage-cor-prf1}, we have
\eqref{eq:coverage-2d} and the proof is completed.
\end{proof}

\subsection{Device discovery}\label{subsec:discovery}

We next consider the problem of device discovery.
Devices in the transmission mode transmit broadcast messages, whereas
a device in the receiving mode can discover the transmitters if it can
successfully decode the broadcast messages.
When a device in the receiving mode receives the signal from one
transmitting device, the signals from all other transmitting devices
work as interference.
Then, the expected number of transmitting devices discovered by the
typical device is represented by
\begin{equation}\label{eq:discovery}
  \mathcal{N}(\theta)
  = (1-p)\,\Exp_\Psi^0\Biggl[
      \sum_{m=1}^\infty
        \ind{\{\SINR_m>\theta\}}
    \Biggr].
\end{equation}

\begin{proposition}\label{prp:discovery}
Consider the D2D network model described in Section~\ref{subsec:D2D}
with the devices deployed according to a stationary PPCP given in
Section~\ref{subsec:PPCP}.
Then, the expected number $\mathcal{N}(\theta)$ of transmitting
devices discovered by the typical device is obtained by
\eqref{eq:coverage} in Theorem~\ref{thm:Coverage} replacing the
integral range $\|z\|>\|y\|$ in~\eqref{eq:C(theta)} by $\R^d$.
Moreover, when $d=2$ and $Q(\dd y) = f_{\os}(\|y\|)\,\dd y$ for
$y\in\R^2$, $\mathcal{N}(\theta)$ is reduced to \eqref{eq:coverage-2d}
in Corollary~\ref{cor:coverage} replacing the integral range
$(s,\infty)$ in \eqref{eq:C(theta)-2d} by $(0,\infty)$.
\end{proposition}  

\begin{proof}
Since
$\Exp_\Psi^0\bigl[\sum_{m=1}^\infty\ind{\{\SINR_m>\theta\}}\bigr] =
\sum_{m=1}^\infty\Prb_\Psi^0(\SINR_m>\theta)$, the difference between
\eqref{eq:coverage-def} and \eqref{eq:discovery} is only the event
$\{\|Y'_m\|\le\|Y'_j\|, j\in\N\}$.
This leads to the difference of the integral ranges in $C_\theta(y,x)$
in \eqref{eq:C(theta)} and in $\Hat{C}_\theta(s,r)$ in
\eqref{eq:C(theta)-2d}.
\end{proof}

\begin{remark}
Since $\Prb_\Psi^0\bigl(\bigcup_{m=1}^\infty\{\SINR_m>\theta\}\bigr)
\le \sum_{m=1}^\infty\Prb_\Psi^0(\SINR_m>\theta) =
\mathcal{N}(\theta)$, Proposition~\ref{prp:discovery} also gives an
upper bound for the coverage probability with the max-SINR association
policy, where a device in the receiving mode receives a message with
the strongest SINR.
This upper bound is known to be exact for $\theta>1$ since
$\sum_{m=1}^\infty\ind{\{\SINR_m>\theta\}} \le 1 + \theta^{-1}$ (see
\cite{DhilGantBaccAndr12} or \cite[Lemma
  5.1.2]{BlasHaenKeelMukh18}).
\end{remark}  

\subsection{Numerical experiments}\label{subsec:numeric}

We present the results of numerical experiments for the analytical
results obtained in Sections~\ref{subsec:coverage}
and~\ref{subsec:discovery}.
We set $d=2$ and the distribution~$Q$ for the location of the
offspring points as $Q(\dd y) = f_{\os}(\|y\|)\,\dd y$ and $f_{\os}(s)
= e^{-s^2/(2\sigma^2)}/(2\pi\sigma^2)$, $s\ge0$; that is, $Q$ is the
isotropic normal distribution with variance~$\sigma^2$, so that the
resulting PPCP~$\Psi$ is the Thomas point process.
Furthermore, the path-loss function is set as $\ell(r)=r^{-\beta}$,
$r>0$, with $\beta>2$.

The numerical results for the coverage probability are given in
Figure~\ref{fig:coverage}, where the values of $\CP(\theta)$ with
different values of $\theta$ and $\sigma^2$ are plotted.
The other parameters are fixed at $\lambda_\Phi=\pi^{-1}$, $\mu=10$,
$p=0.5$, $\beta=4$ and $N=0$.
For comparison, the values when the devices are located according to a
homogeneous PPP are also displayed in the figure with the label
``$\sigma^2\to\infty$.''
From Figure~\ref{fig:coverage}, we can see that, as the value of
$\sigma^2$ increases, the coverage probability decreases and is closer
to that with the homogeneous PPP.
This is contrary to the case of cellular networks, where the coverage
probability increases and is closer to that with the homogeneous PPP
from below as the variance of the locations of offspring points
increases (see \cite{Miyo19}).
This difference is thought to be due to the fact that the locations of
a receiving device and its communication partner are near to each
other in the PPCP-deployed D2D network since they are both points of
the same PPCP, whereas the location of a receiver is likely far
from that of the associated BS in the PPCP-deployed cellular network
since their locations are independent of each other.

\begin{figure}
\begin{center}
\vbox{\kern2cm\hbox{\kern-.3cm\includegraphics[width=12cm]{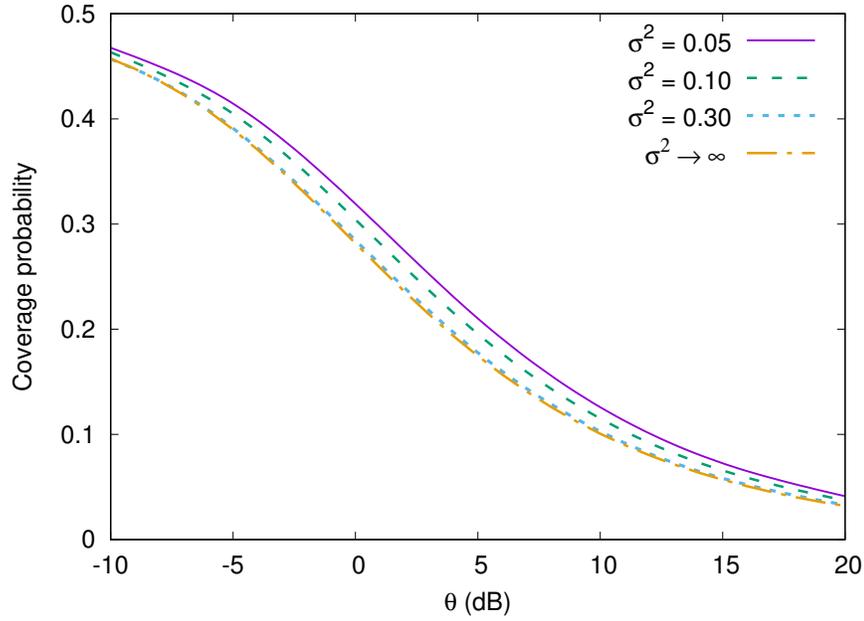}}\kern-2cm}%
\end{center}
\caption{Coverage probability as a function of SINR
  threshold ($\lambda_\Phi=\pi^{-1}$, $\mu=10$, $p=0.5$, $\beta=4$ and
  $N=0$)}\label{fig:coverage}
\end{figure}

The results of the device discovery is given in
Figure~\ref{fig:discovery}, where we know that the closed form
expression of the expected number of discovered devices is obtained as
$\mathcal{N}^{(\mathsf{PPP})}(\theta) =
(1-p)(\beta/2\pi)\sin(2\pi/\beta)\,\theta^{-2/\beta}$ for the case of
the homogeneous PPP with $N\equiv0$ (see, e.g.,
\cite{HamiChelBussFleu08}).
The figure shows similar features to the coverage probability.

\begin{figure}
\begin{center}
\vbox{\kern2cm\hbox{\kern-.3cm\includegraphics[width=12cm]{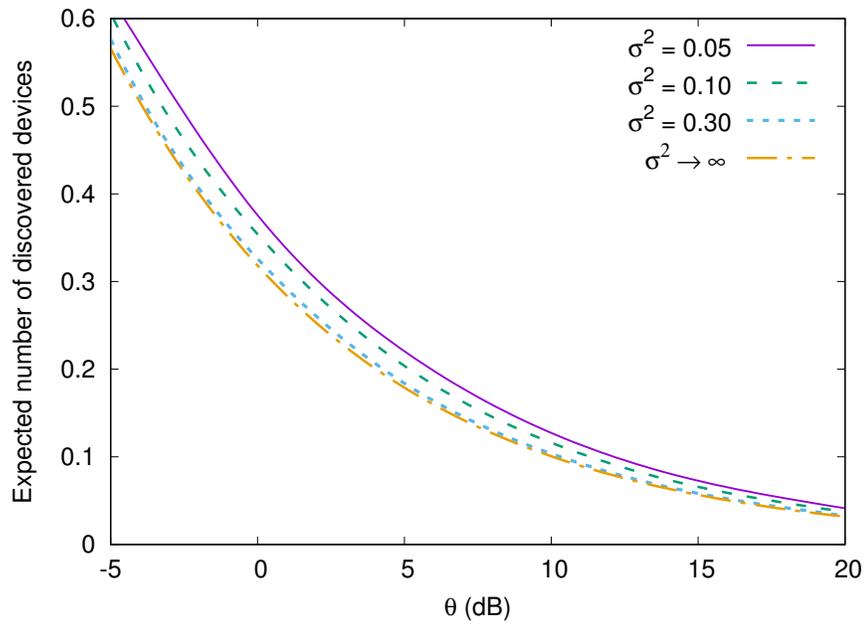}}\kern-2cm}%
\end{center}
\caption{Expected number of discovered devices as a function of SINR
  threshold ($\lambda_\Phi=\pi^{-1}$, $\mu=10$, $p=0.5$, $\beta=4$ and 
  $N=0$)}\label{fig:discovery}
\end{figure}

\section{Conclusion}\label{sec:concl}

In this paper, we have presented an alternative form of Neveu's
exchange formula for jointly stationary point processes on $\R^d$ and
then demonstrated that it is useful for the analysis of spatial
stochastic models given based on stationary PPCPs.
We have first applied it to the Palm characterization for a stationary
PPCP and then to the analysis of a D2D network modeled using a
stationary PPCP.
Although we have only considered some fundamental problems, we expect
that the new form of the exchange formula will be utilized for the
analysis of more sophisticated models leading up to the development of
5G and beyond networks.

\bibliographystyle{plain}

\end{document}